\documentclass{amsart}
\usepackage{amsmath,amssymb,amsthm,enumerate,amscd}
\usepackage{url}
\usepackage{array}

\newtheorem{theorem}{Theorem}[section]
\newtheorem{lemma}[theorem]{Lemma}
\newtheorem{proposition}[theorem]{Proposition}

\newtheorem{definition}[theorem]{Definition}
\newtheorem{rem}[theorem]{Remark}
\newtheorem{preremark}[theorem]{Remark}
	\newenvironment{remark}%
		{\begin{preremark}\rm}{\end{preremark}}
\newtheorem{preexample}[theorem]{Example}
		{\begin{preexample}\rm}{\end{preexample}}
\newcommand{\QQ}{\mathbb{Q}}
\newcommand{\ZZ}{\mathbb{Z}}
\newcommand{\FF}{\mathbb{F}}

\newcommand{\NN}{\mathbb{N}}
\newcommand{\OK}{\mathcal{O}_K}

\renewcommand{\O}{\mathcal{O}}
\newcommand{\frakp}{\mathfrak{p}}
\newcommand{\norm}{\operatorname{Norm}}
\newcommand{\disc}{\operatorname{disc}}
\newcommand{\End}{\operatorname{End}}

\newcommand{\ord}{\operatorname{ord}}

\def\Z{\mathbb{Z}}
\def\Q{\mathbb{Q}}
\def\F{\mathbb{F}}

\def\P{\mathbb{P}}

\def\Fp{\F_p}

\def\isom{\xrightarrow{\sim}}

\title[Deterministic elliptic curve primality proving]{Deterministic elliptic curve primality proving for a special sequence of numbers}
\author[A. Abatzoglou]{Alexander Abatzoglou}
\address{Department of Mathematics, University of California, Irvine, CA 92697}
\email{aabatzog@math.uci.edu}
\author[A. Silverberg]{Alice Silverberg}
\address{Department of Mathematics, University of California, Irvine, CA 92697}
\email{asilverb@math.uci.edu}
\author[A.V. Sutherland]{\\Andrew V.\ Sutherland}
\address{Department of Mathematics, MIT, Cambridge, MA 02139}
\email{drew@math.mit.edu}
\author[A. Wong]{Angela Wong}
\address{Department of Mathematics, University of California, Irvine, CA 92697}
\email{awong@math.uci.edu}
\thanks{This work was supported by the National Science Foundation under grants CNS-0831004 and DMS-1115455.}

\begin{document}

\begin{abstract}
We give a deterministic algorithm that very quickly proves the primality or compositeness of the integers $N$ in a certain sequence, using an elliptic curve $E/\Q$ with complex multiplication by the ring of integers of $\Q(\sqrt{-7})$.
The algorithm uses $O(\log N)$ arithmetic operations in the ring $\Z/N\Z$, implying a bit complexity that is quasi-quadratic in $\log N$.
Notably, neither of the classical ``$N-1$" or ``$N+1$" primality tests apply to the integers in our sequence.
We discuss how this algorithm may be applied, in combination with sieving techniques, to efficiently search for very large primes.
This has allowed us to prove the primality of several integers with more than 100,000 decimal digits, the largest of which has more than a million bits in its binary representation.  
%We believe that this is
At the time it was found, it was the largest proven prime $N$ for which no significant partial factorization of $N-1$ or $N+1$ is known (as of final submission it was second largest).
\end{abstract}

\maketitle

\section{Introduction}
With the celebrated result of Agarwal, Kayal, and Saxena \cite{aks04}, one can now unequivocally determine the primality or compositeness of any integer in deterministic polynomial time.  With the improvements of Lenstra and Pomerance \cite{lenstrapomerance11}, the AKS algorithm runs in $\tilde{O}(n^6)$ time, where $n$ is the size of the integer to be tested (in bits).  However, it has long been known that for certain special sequences of integers, one can do much better.  The two most famous examples are the Fermat numbers $F_k=2^{2^k}+1$, to which one may apply P\'epin's criterion \cite{pepin1877}, and the Mersenne numbers $M_p=2^p-1$, which are subject to the Lucas-Lehmer test~\cite{lehmer30}.  In both cases, the corresponding algorithms are deterministic and run in $\tilde{O}(n^2)$ time.

In fact, every prime admits a proof of its primality that can be verified by a deterministic algorithm in $\tilde{O}(n^2)$ time.
Pomerance shows in \cite{pomerance87} that for every prime $p > 31$ there exists an elliptic curve $E/\Fp$ with an $\Fp$-rational point $P$ of order $2^r > (p^{1/4}+1)^2$, which allows one to establish the primality of $p$ using just $r$ elliptic curve group operations.
Elliptic curves play a key role in Pomerance's proof; the best analogous result using classical primality certificates yields an $\tilde{O}(n^3)$ time bound \cite{pratt75}, cf.\ \cite[Thm.\ 4.1.9]{crandal05}.

The difficulty in applying Pomerance's result lies in finding the pair $(E,P)$, a task for which no efficient method is currently known.
Rather than searching for suitable pairs $(E,P)$, 
we instead fix a finite set of curves $E_a/\Q$, each equipped with a known rational point~$P_a$ of infinite order.
To each positive integer $k$ we associate one of the curves $E_a$ and define an integer~$J_k$ for which we give a necessary and sufficient condition for primality: $J_k$ is prime if and only if the reduction of $P_a$ in $E_a(\Fp)$ has order $2^{k+1}$ for every prime $p$ dividing~$J_k$.
Of course $p=J_k$ when $J_k$ is prime, but this condition can easily be checked without knowing the prime factorization of $J_k$.
This yields a deterministic algorithm that runs in $\tilde{O}(n^2)$ time (see Algorithm~\ref{algorithm}).

Our results extend the methods used by Gross~\cite{gross04}, Denomme and Savin~\cite{denommesavin08},  Tsumura~\cite{tsumura11}, and Gurevich and Kunyavski\u{\i} \cite{gurkun12},
all of which fit within a general framework laid out by Chudnovsky and Chudnovsky in~\cite{chudnovsky86} for determining the primality of integers in special sequences using elliptic curves with complex multiplication (CM).
The elliptic curves that we use lie in the family of quadratic twists defined by the equations
\begin{equation}
\label{twist_weierstrass}
E_a: y^2 = x^3 - 35a^2x - 98a^3,
\end{equation}
for square-free integers $a$ such that $E_a(\Q)$ has positive
rank. Each curve has
good reduction outside of $2$,  $7$, and the prime divisors of $a$, and has
CM by $\Z[\alpha]$, 
where 
$$\alpha = \frac{1 + \sqrt{-7}}{2}.$$
For each curve $E_a$, we fix a point $P_a\in E_a(\Q)$ of infinite order with $P_a\not\in2E_a(\Q)$.

For each positive integer $k$, let
$$
j_k = 1 + 2\alpha^k\in\Z[\alpha],  \qquad J_k = j_k\bar{j_k} = 1 + 2(\alpha^k + \bar{\alpha}^k) + 2^{k+2}
\in\NN.
$$
The integer sequence $J_k$ satisfies the linear recurrence relation
$$
J_{k+4} = 4J_{k+3}-7J_{k+2}+8J_{k+1}-4J_k,
$$
with initial values $J_1=J_2=11$, $J_3=23$, and $J_4=67$.
Then (by~Lemma~\ref{divislemma}) $J_k$ is composite for $k\equiv 0 {\pmod 8}$ and for $k\equiv 6\pmod{24}$.
To each other value of $k$ we assign a squarefree integer $a$, based on the congruence class of $k\pmod{72}$, as listed in Table~\ref{table:twistspts}.
Our choice of $a$ is based on two criteria.
First, it ensures that when $J_k$ is prime, the Frobenius endomorphism of $E_a\bmod J_k$ corresponds to complex multiplication by $j_k$ (rather than $-j_k$) and 
$$
E_a(\Z/J_k\Z)\simeq \Z/2\Z\times\Z/2^{k+1}\Z.
$$
Second, it implies that when $J_k$ is prime, the reduction of the point $P_a$ has order $2^{k+1}$ in $E(\Z/J_k\Z)$.
The second condition is actually stronger than necessary (in general, one only needs $P_a$ to have order greater than $2^{k/2+1}$), but it simplifies matters.
Note that choosing a sequence of the form $j_k = 1+\Lambda_k$ means
that $E_a(\Z[\alpha]/(j_k)) \simeq \Z[\alpha]/\Lambda_k$, whenever $J_k$ is prime and 
$j_k$ is the Frobenius endomorphism of 
$E_a$ mod $J_k$ (see Lemma \ref{structure1}).

We prove in Theorem~\ref{mainthm} that the integer $J_k$ is prime if and only if the point~$P_a$ has order $2^{k+1}$ on ``$E_a\bmod J_k$".  More precisely, we prove that if one applies the standard formulas for the elliptic curve group law to compute scalar multiples $Q_i=2^iP_a$ using projective coordinates $Q_i=[x_i,y_i,z_i]$ in the ring $\ZZ/J_k\ZZ$, then $J_k$ is prime if and only if $\gcd(J_k,z_k)=1$ and $z_{k+1}=0$.
This allows us to determine whether~$J_k$ is prime or composite using $O(k)$ operations in the ring $\ZZ/J_k\ZZ$, yielding a bit complexity of $O(k^2\log k\log\log k)=\tilde{O}(k^2)$ (see Proposition~\ref{prop:complexity} for a more precise bound).

We note that, unlike the Fermat numbers, the Mersenne numbers, and many similar numbers of a special form, the integers $J_k$ are not amenable to any of the classical ``$N-1$" or ``$N+1$" type primality tests (or combined tests) that are typically used to find very large primes (indeed, the 500 largest primes currently listed in~\cite{caldwell12} all have the shape $ab^n\pm 1$ for some small integers $a$ and $b$).

In combination with a sieving approach described in \S \ref{section:computations}, we have used our algorithm to determine the primality of $J_k$ for all $k\le 1.2\times 10^6$. 
The prime values of~$J_k$ are listed in Table~\ref{table:primes}.
At the time it was found, the prime  $J_{1,111,930}$, which has 334,725
decimal digits, 
was the largest proven prime $N$ for which no significant partial factorization of either $N-1$ or $N+1$ was known. On July 4, 2012 it was superseded by a 377,922 digit prime found by
David Broadhurst \cite{Broadhurst} for which no significant factorization of $N-1$ or $N+1$ is known;
Broadhurst constructed an ECPP primality proof for this prime, but it is not a Pomerance proof.

Generalizations have been suggested to the settings
of higher dimensional abelian varieties with complex multiplication,
algebraic tori, and group schemes by
Chudnovsky and Chudnovsky \cite{chudnovsky86}, Gross \cite{gross04},
and  Gurevich and Kunyavski\u{\i} \cite{gurkun}, respectively.
In the PhD theses of the first and fourth authors, and in a forthcoming paper, 
we are extending the results in this paper to a more general framework.
In that paper we will also 
explain why, when restricting to elliptic curves over $\Q$, this method 
requires curves with CM
by $\Q(\sqrt{-D})$ with $D=$ 1, 2, 3, or 7.

\smallskip 

\noindent{\bf{Acknowledgments:}} 
We thank Daniel J.~Bernstein, Fran\c{c}ois Morain, 
Carl Pomerance, and Karl Rubin for helpful conversations,
and the organizers of ECC 2010, the First Abel Conference, and 
the AWM Anniversary Conference 
where useful discussions took place.
We thank the reviewers for helpful comments.
We also thank Henri Cohen and Richard Pinch for helpful comments given
at ANTS-X.

\section{Relation to Prior Work}
\label{PriorWork}
In \cite{chudnovsky86}, Chudnovsky and Chudnovsky consider certain sequences of integers $s_k = \norm_{K/\Q}(1+\alpha_0\alpha_1^k)$, defined by algebraic integers $\alpha_0$ and $\alpha_1$ in an imaginary quadratic field $K=\Q(\sqrt{D})$.
They give sufficient conditions for the primality of $s_k$, using an elliptic curve $E$ with CM by $K$.
In our setting, $D=-7$, $\alpha_0=2$, $\alpha_1=(1+\sqrt{-7})/2$, and $J_k=s_k$.
The key difference here is that we give necessary and sufficient criteria for primality that can be efficiently checked by a deterministic algorithm.
This is achieved by carefully selecting the curves $E_a/\Q$ that we use, so that in each case we are able to prove that the point $P_a\in E_a(\Q)$ reduces to a point of maximal order $2^{k+1}$ on $E_a\bmod J_k$, whenever $J_k$ is prime.
Without such a construction, we know of no way to obtain \emph{any} non-trivial point on $E\bmod s_k$ in deterministic polynomial time.

Our work is a direct extension of the techniques developed by Gross~\cite{gross04,YanJames}, Denomme and Savin~\cite{denommesavin08}, Tsumura \cite{tsumura11}, and 
 Gurevich and Kunyavski\u{\i} \cite{gurkun12}, who use elliptic curves with CM by the ring of integers of $\Q(i)$ or $\Q(\sqrt{-3})$ to test the primality of Mersenne, Fermat, and related numbers.
However, as noted by Pomerance \cite[\S 4]{pomerance10}, the integers considered in \cite{denommesavin08} can be proved prime using classical methods that are more efficient and do not involve elliptic curves, and the same applies to \cite{gross04,tsumura11,YanJames,gurkun12}.
But this is not the case for the sequence we consider here.

\section{Background and Notation}
\label{notation}

\subsection{Elliptic curve primality proving}
Primality proving algorithms based on elliptic curves have been proposed since the mid-1980s.
Bosma \cite{bosma85} and Chudnovsky and Chudnovsky \cite{chudnovsky86} considered a setting similar to the one employed here, using elliptic curves to prove the primality of numbers of a special form; Bosma proposed the use of elliptic curves with complex multiplication by $\mathbb Q(i)$ or $\mathbb Q(\sqrt{-3})$, while Chudnovsky and Chudnovsky considered a wider range of elliptic curves and other algebraic varieties.
Goldwasser and Kilian \cite{goldwasserkilian86,goldwasserkilian99} gave the first general purpose elliptic curve primality proving algorithm, using randomly generated elliptic curves.
Atkin and Morain \cite{atkinmorain93,MorainTitanic} developed an improved version of the Goldwasser-Kilian algorithm that uses the CM method to construct the elliptic curves used, rather than generating them at random (it does rely on probabilistic methods for root-finding).
With asymptotic improvements due to Shallit, the Atkin-Morain algorithm has a heuristic expected running time of $\tilde{O}(n^4)$, which makes it the method of choice for general purpose primality proving \cite{morain07}.
%We also note that 
Gordon \cite{gordon1989} proposed a general purpose compositeness test using supersingular reductions of CM elliptic curves over~$\Q$.

Throughout this paper, if $E \subset \P^2$ is an elliptic curve over $\QQ$, we shall write points $[x,y,z]\in E(\QQ)$ so that $x,y,z \in \ZZ$ and $\gcd(x,y,z) = 1$, and we may use $(x,y)$ to denote the projective point $[x,y,1]$.

We say that a point $P=[x,y,z]\in E(\Q)$ is \emph{zero mod} $N$ when $N$ divides $z$; otherwise $P$ is \emph{nonzero mod} $N$.
Note that if $P$ is zero mod $N$ then $P$ is zero mod~$p$ for all primes $p$ dividing $N$.

\begin{definition}
\label{stronglynonzero}
Given an elliptic curve $E$ over $\QQ$, a point $P = [x,y,z] \in E(\QQ)$, and $N\in\Z$, 
we say that $P$ is \em{strongly nonzero} mod $N$ if $\gcd(z,N) = 1$.
\end{definition}
\noindent
If $P$ is strongly nonzero mod $N$, then $P$ is nonzero mod $p$ for every prime $p|N$, and
if~$N$ is prime, then $P$ is strongly nonzero mod $N$ if and only if $P$ is nonzero mod~$N$.

We rely on the following fundamental result, which can be found in \cite{goldwasserkilian86,LenstraICM,goldwasserkilian99}.

\begin{proposition}
\label{prop:ECPP2}
Let $E/\Q$ be an elliptic curve, let $N$ be a positive integer prime to $\disc(E)$,
let $P\in E(\Q)$, and let $m > (N^{1/4}+1)^2$.
Suppose $mP$ is zero mod~$N$ and $(m/q)P$ is strongly nonzero mod~$N$ for all primes $q|m$.
Then $N$ is prime.
\end{proposition}

To make practical use of Proposition~\ref{prop:ECPP2}, one needs to know the prime factorization of $m$.
For general elliptic curve primality proving this presents a challenge; the algorithms of Goldwasser-Kilian and Atkin-Morain use different approaches to ensure that $m$ has an easy factorization, but both must then recursively construct primality proofs for the primes $q$ dividing $m$.
In our restricted setting we effectively fix the prime factorization of $m=2^{k+1}$ ahead of time.

Next we give a variant of Proposition~\ref{prop:ECPP2} that replaces 
``strongly nonzero''  with ``nonzero'', at the expense of 
$m$ being a prime power with a larger lower bound.

\begin{proposition}
\label{prop:ECPP2variant}
Let $E/\Q$ be an elliptic curve, let $p$ be a prime, let
$N$ be an odd positive integer prime to $p\disc(E)$,
and let $P\in E(\Q)$.
Suppose $b$ is a positive integer such that
$p^b > (\sqrt{N/3} + 1)^2$ and 
$p^b P$ is zero mod~$N$ and
$p^{b-1} P$ is nonzero mod~$N$.
Then $N$ is prime.
\end{proposition}
\begin{proof}
Since $p^{b-1} P$ is nonzero mod~$N$, there are a prime divisor $q$ of $N$ and
a positive integer $r$ such that $q^r$ exactly divides $N$ and $p^{b-1} P$ is nonzero mod $q^r$. 
Let $E_1(\Z/q^r\Z)$ denote the kernel of the reduction map
$E(\Z/q^r\Z) \to E(\F_q)$.
It follows, for example, from  
\cite[Thm.~4.1]{Milne} that $E_1(\Z/q^r\Z)$ is a $q$-group.
Let $P' \in E(\Z/q^r\Z)$ be the reduction of $P$ mod $q^r$ and let $P''$ be the image of $P'$ in $E(\F_q)$.
If $p^{b-1} P'' = 0$ then $p^{b-1} P' \in E_1(\Z/q^r\Z)$, so $p^{b-1} P'$ has order a power of $q$. But by assumption it has order $p$, which is prime to $N$. This is a contradiction, so $P''$ has order $p^b$. 
If $N$ were composite, then $q \le N/3$ since $N$ is odd, so by the Hasse bound,
$$
p^b \le |E(\F_q)| \le (\sqrt{q} + 1)^2 \le (\sqrt{N/3} + 1)^2,
$$
contradicting the hypothesis that $p^b > (\sqrt{N/3} + 1)^2$.
\end{proof}

\subsection{Complex multiplication and Frobenius endomorphism}
For any number field $F$, let $\O_F$ denote its ring of integers.
If $E$ is an elliptic curve over a field~$K$, and $\Omega_K$ is the space of
holomorphic differentials on $E$ over $K$, then $\Omega_K$ is a one-dimensional
$K$-vector space, and there is a canonical ring homomorphism
\begin{equation}
\label{endhom}
\End_K(E) \to \End_K(\Omega) = K.
\end{equation}
Suppose now that $E$ is an elliptic curve over an imaginary quadratic field $K$,
and that $E$ has complex multiplication (CM) by $\O_K$, meaning that $\End_K(E) \simeq \O_K$.
Then the image of the map in \eqref{endhom} is $\O_K$. Let
$\psi : \O_K \to \End_K(E)$ denote the inverse map.
Suppose that $\frakp$ is a prime ideal of $K$ at which $E$ has good reduction 
and let $\tilde{E}$ denote the reduction of $E$ mod $\frakp$. Then the composition
$$
\O_K \isom \End_K(E) \hookrightarrow \End_{\O_K/\frakp}(\tilde{E}),
$$
where the first map is $\psi$ and the second is induced by reduction mod $\frakp$,
gives a canonical embedding 
\begin{equation}
\label{canonident}
 \O_K\hookrightarrow\End(\tilde{E}).
\end{equation}
The Frobenius endomorphism of $\tilde{E}$ is  
$(x,y) \mapsto (x^{q},y^{q})$ where $q=\norm_{K/\Q}(\frakp)$;
under the  embedding in \eqref{canonident}, the Frobenius endomorphism is 
the image of a particular generator $\pi$ of the (principal) ideal $\frakp$.
By abuse of notation, we say that the Frobenius endomorphism is $\pi$.

\section{Main Theorem}

In this section we state and prove our main result, Theorem \ref{mainthm}, 
which gives  a necessary and sufficient condition for the primality of
the numbers $J_k$.

Fix a particular square root of ${-7}$ and let $K = \QQ(\sqrt{-7})$.
Let
$$
\alpha = \frac{1+\sqrt{-7}}{2} \in \O_K,
$$
and for each positive integer $k$, let
$$
j_k = 1 + 2\alpha^k \in\Z[\alpha] \qquad\text{and} \qquad J_k = \norm_{K/\QQ}(j_k) = j_k\bar{j_k} \in \NN.
$$
Note that $J_k$ is prime in $\ZZ$ if and only if $j_k$ is prime in~$\OK$.
Note also that $\norm_{K/\QQ}(\alpha) = \alpha\bar{\alpha} = 2$.

Recall the family of elliptic curves $E_a$ defined by  \eqref{twist_weierstrass}.
Lemma \ref{divislemma} below shows that $J_k$ is composite if
$k\equiv 0 \pmod{8}$ or
$k\equiv 6 \pmod{24}$,
so we omit these cases from our primality criterion.
For each remaining value of $k$, 
Table~\ref{table:twistspts} lists the twisting parameter $a$ and the point $P_a\in E_a(\Q)$ we associate to $k$.
For each of these $a$, the elliptic curve 
$E_a$ has rank one over $\Q$, and the point $P_a$ is a generator for $E_a(\Q)$ modulo torsion.
 
\begin{table}[ht]
  \caption{The twisting parameters $a$ and points $P_a$}\label{table:twistspts}
    \begin{tabular}{ l r r  }
    $k$  & \hspace{40pt}$a$ & \hspace{70pt}$P_a$   \\ 
    \hline
$k \equiv 0$ or $ 2 \pmod{3}$     & $-1$ & $(1,8)$  \\ 
$k \equiv 4,7,13,22   \pmod{24}$  & $-5$ & $(15,50)$  \\
$k \equiv 10  \pmod{24}$   & $-6$ & $(21,63)$  \\
$k \equiv  1,19,49,67 \pmod{72}$       
         & $-17$ & $(81,440)$  \\
$k \equiv  25, 43 \pmod{72}$       
         & $-111$ & $(-633, 12384)$  \\\hline
     \end{tabular}
\end{table}

\begin{theorem}
\label{mainthm}
Fix $k > 1$ such that $k \not\equiv 0 \pmod{8}$ and $k \not\equiv 6 \pmod{24}$.
Let $P_a \in E_a(\Q)$ be as in Table~\ref{table:twistspts} (depending on $k$).
The following are equivalent:
\begin{enumerate}
\item $2^{k+1}P_a$ is zero mod $J_k$ and $2^kP_a$ is strongly nonzero mod $J_k$;
\item  $J_k$ is prime.
\end{enumerate} 
\end{theorem}

\begin{remark}\label{rem:nonzero}
Applying Proposition \ref{prop:ECPP2variant} with $N=J_k$, $p=2$, and $b=k+1$,
we can add an equivalent condition in Theorem~\ref{mainthm} as long as $k \ge 6$, namely:
\begin{enumerate}
\item[(iii)] $2^{k+1}P_a$ is zero mod $J_k$ and $2^kP_a$ is nonzero mod $J_k$.
\end{enumerate} 
\end{remark}

We shall prove Theorem~\ref{mainthm} via a series of lemmas, but let us first outline the proof.
One direction is easy: since $2^{k+1} > (J_k^{1/4}+1)^2$ for all $k>1$, if~(i) holds then so does (ii), by Proposition~\ref{prop:ECPP2} 
(where the hypothesis $\gcd(J_k,\disc(E_a))=1$ holds by Lemma~\ref{divislemma} below).

Now fix $a$ and $P_a$ as in Table~\ref{table:twistspts}, and let
$\tilde{P}_a$ denote the reduction of $P_a$ modulo~$j_k$.
We first compute a set $S_a$ such that if $k\in S_a$ and $j_k$ is prime, then   
$E_a(\OK/(j_k))\simeq\OK/(2\alpha^k)$ as $\OK$-modules.
We then compute a set $T_a$ such that if $k\in T_a$ and $j_k$ is prime, then
$\tilde{P}_a$ does not lie in $\alpha E_{a}(\OK/(j_k))$ if and only if $k \in T_{a}$ (note that $\alpha\in\OK\hookrightarrow\End(E_a)$).
For $k \in S_a \cap T_a$, the point $\tilde{P}_a$ has order $2^{k+1}$ whenever $J_k$ is prime.

We now fill in the details.
Many of the explicit calculations below were performed with the assistance of the Sage computer algebra system \cite{sage}.
\subsection{The linear recurrence sequence $J_k$}
As noted in the introduction, the sequence $J_k$ satisfies the linear recurrence relation
\begin{equation}
\label{recurrence}
J_{k+4} = 4J_{k+3}-7J_{k+2}+8J_{k+1}-4J_k.
\end{equation}
We now prove this, and also note some periodic properties of this sequence.
See \cite{epsw} or \cite[Ch.\ 6]{lnff} for basic properties of linear recurrence sequences.

\begin{definition}
\label{periodicdef}
We call a sequence $a_k$ (purely) \emph{periodic} if there exists an
integer $m$ such that $a_k=a_{k+m}$ for all $k$.  The minimal such $m$ is
the \emph{period} of the sequence.
\end{definition}

\begin{lemma}
\label{lemma:recurrence}
The sequence $J_k$ satisfies \eqref{recurrence}.
If $p$ is an odd prime and $\frakp\subset\OK$ is a prime ideal above $(p)$, then
the sequence $J_k\bmod p$ is periodic, with period equal to the least common multiple of the orders of $2$ and $\alpha$ in~$(\OK/\frakp)^*$.
\end{lemma}
\begin{proof}
The characteristic polynomial of the linear recurrence in \eqref{recurrence} is
\[
f(x) = x^4-4x^3+7x^2-8x+4 = (x-1)(x-2)(x^2-x+2),
\]
whose roots are $1, 2, \alpha$, and $\bar{\alpha}$.  It follows that the sequences $1^k$, $2^k$, $\alpha^k$, and $\bar{\alpha}^k$, and any linear combination of these sequences, satisfy \eqref{recurrence}.  Thus $J_k$ satisfies \eqref{recurrence}.

One easily checks that the lemma is true for $p=7$, so assume $p\ne 7$.
Let $A$ be the $4\times 4$ matrix with $A_{i,j} = J_{i+j-1}$.
Then $\det A=-2^{12}\cdot7$ is nonzero mod $p$, hence its rows are linearly independent over $\Fp$.
It follows from Theorems 6.19 and 6.27 of \cite{lnff} that the sequence $J_k\bmod p$ is periodic, with
period equal to the lcm of the orders of the roots of $f$ in $\bar{\FF}_p^*$ (which we note are distinct).
These roots all lie in $\OK/\frakp\simeq \FF_{p^d}$, where $d\in\{1,2\}$ is the residue degree of~$\frakp$.
Since $\bar{\alpha} = 2/\alpha$, the order of $\bar{\alpha}$ in $(\OK/\frakp)^*$ divides the
lcm of the orders of $2$ and $\alpha$.  The lemma follows.
\end{proof}

When $p$ is an odd prime, let $m_p$ denote the period of the sequence $J_k\bmod p$.
Lemma~\ref{lemma:recurrence} implies that $m_p$ always divides $p^2-1$, and it divides $p-1$ whenever $p$ splits in $K$.

\begin{lemma}
\label{divislemma}
The following hold:
\begin{enumerate}
\item $J_k$ is divisible by $3$ if and only if $k \equiv 0 \pmod{8}$;
\item $J_k$ is divisible by $5$ if and only if $k \equiv 6 \pmod{24}$;
\item $J_k \equiv 2 \pmod{7}$ if $k \equiv 0 \pmod{3}$, and 
		$J_k \equiv 4 \pmod{7}$ otherwise;
\item for  $k>1$,
$J_k\equiv 3 \pmod{8}$ if $k$ is even, and $J_k\equiv 7 \pmod{8}$ if $k$ is odd;
\item $J_k$ is divisible by $17$ if and only if $k \equiv 54 \pmod{144}$;
\item $J_k$ is not divisible by $37$.
\end{enumerate}
\end{lemma}
\begin{proof}
Lemma \ref{lemma:recurrence} allows us to compute the periods $m_3=8$, $m_5=24$, $m_7=3$, 
$m_{17} = 144$, and $m_{37} = 36$.  It then suffices to check, for $p=3, 5, 17$, and 37, 
when $J_k\equiv 0\pmod{p}$ for $1\le k\le m_p$, and to determine the values of
$J_k\pmod{7}$ for $1\le k\le 3$.

It is easy to check that $\alpha^{k} + \bar{\alpha}^{k} \equiv 3 \pmod{4}$
for odd $k>1$, and  $\alpha^{k} + \bar{\alpha}^{k} \equiv 1 \pmod{4}$ otherwise.
Since $J_k = 1 + 2(\alpha^k + \bar{\alpha}^k) + 2^{k+2}$, we have (iv).

As an alternative proof for one direction of (i,ii), note that $\alpha$ and $\bar{\alpha}$ each has order $8$ in $(\OK/(3))^\times$. Hence if $k \equiv 0 \pmod{8}$, then $J_k = 1 + 2(\alpha^k + \bar{\alpha}^k) + 2^{k+2} \equiv 1 + 2(1 + 1) + 1 \equiv 0 \pmod{3}$. Similarly, $\alpha^6 \equiv 2 \equiv \bar{\alpha}^6 \pmod{5}$, so 
$J_k \equiv 1 + 2(4) + 1 \equiv 0 \pmod{5}$ when $k \equiv 6 \pmod{24}$.

\end{proof}

\subsection{The set $\boldsymbol{S_a}$}
For each squarefree integer $a$ we define the set of integers
$$
S_a := \Bigl\{ k > 1 : \left(\frac{a}{J_k}\right)\left(\frac{j_k}{\sqrt{-7}}\right) = 1 \Bigr\},
$$ 
where $\left(\frac{\,\,\,}{\,\,\,}\right)$ denotes the (generalized) Jacobi symbol.

If $j_k$ is prime in $\OK$, then 
the Frobenius endomorphism of $E_a$ over the finite field $\OK/(j_k)$ corresponds to either $j_k$ or $-j_k$.
For elliptic curves over~$\Q$ with complex multiplication, one can easily determine which is the case.

\begin{lemma}
\label{structure1}
Suppose $a$ is a squarefree integer, $k > 1$, and $j_k$ is prime in $\OK$.
Then:
\begin{enumerate}
\item $k \in S_a$ if and only if
the Frobenius endomorphism of $E_a$ over the finite field $\OK/(j_k)$ is $j_k$;
\item if $k \in S_a$, then $E_a(\OK/(j_k)) \simeq \OK/(2\alpha^k)$ as $\OK$-modules.
\end{enumerate}
\end{lemma}

\begin{proof}
The elliptic curve  $E_{a}$ is the curve in Theorem 1 of \cite[p.~1117]{stark}, with $D=-7$ and $\pi=j_k$.
By \cite[p.~1135]{stark}, the Frobenius endomorphism of  $E_{a}$ over $\OK/(j_k)$ is
$$\left(\frac{a}{J_k}\right)\left(\frac{j_k}{\sqrt{-7}}\right)j_k \in \OK.$$
Part (i) then follows from the definition of $S_a$.
For (ii), note that (i) implies that if $k \in S_a$, then 
$$
E_{a}(\OK/(j_k)) \simeq \ker(j_k - 1) = \ker(2\alpha^k) \simeq  \OK/(2\alpha^k),
$$
which completes the proof.
\end{proof}

The next lemma follows directly from Lemma \ref{divislemma}(iv).

\begin{lemma}
\label{Jacobicomputations}
If $k > 1$, then
\begin{enumerate}
\item $\textstyle{\left(\frac{-1}{J_k}\right) = -1,}$
\item 
$
\left(\frac{2}{J_k}\right) = \left\{
	\begin{array}{rl}
	1 & \text{if $k$ is odd},\\ 
				-1 & \text{if $k$ is even.}
	\end{array}
\right.
$
\end{enumerate}
\end{lemma}

%\begin{proof}
%For  $k>1$,
%$J_k\equiv 3 \pmod{8}$ if $k$ is even, and $J_k\equiv 7 \pmod{8}$ if $k$ is odd.
%\end{proof}

We now explicitly compute the sets $S_a$ for the values of $a$ used in Theorem~\ref{mainthm}.

\begin{table}[ht]
\caption{The sets $S_a$}
\label{table:Sa}
\begin{tabular}{rrll}
$a$&\hspace{6pt}$m$&&$S_a = \{ k > 1 : k\bmod{m} \text{ is as below$\}$}$\\\hline
$-1$ & $3$ && $0, 2$\\
$-5$ & $24$ && $0,2,4,5,7,9,12,13,16,18,21,22,23$\\
$-6$ & $24$ && $3,7,9,10,11,12,13,17,20,22$\\
$-17$ & $144$ && $0,1,5,7,9,10,13,14,15,18,19,20,22,23,27,30,31,33,34,$\\
&&& $36,42,43,44,45,49,50,53,56,61,62,63,66,67,68,70,71,$\\
&&& $72,73,75,76,78,79,80,81,82,83,90,91,92,93,97,99,100,$\\
&&& $104,106,108,110,111,112,114,117,118,121,122,123,125,$\\
&&& $126,128,129,133,135,136,137,138,139,141,143$\\
$-111$ & $7$2 && $2,4,6,9,14,15,18,20,22,23,25,30,33,34,35,37,38,39,41,$\\
&&& $42,43,47,49,50,52,53,54,55,57,58,63,65,66,67,68,70$\\\hline
\end{tabular}
\end{table}

\begin{lemma}
\label{structure2}
For $a\in\{-1,-5,-6,-17,-111\}$ the sets $S_a$ are as in Table~\ref{table:Sa}.
\end{lemma}

\begin{proof}
Since $j_k = 1 + 2\alpha^k$, and $\alpha  \equiv 4 \pmod{\sqrt{-7}}$,
and $2^3 \equiv 1 \pmod{7}$, we have
\begin{equation*}
\left(\frac{j_k}{\sqrt{-7}}\right) = \left(\frac{1 + 2^{2k+1}}{7}\right) = \left\{
\begin{array}{rl}
1 & \text{if } k \equiv 1 \pmod{3},\\
-1 & \text{if } k \equiv 0,2 \pmod{3}.
\end{array}
\right.
\end{equation*}
We now need to compute $(\frac{a}{J_k})$ for $a=-1, -5, -6, -17$, and $-111$.
By Lemma~\ref{Jacobicomputations}(i), we have $(\frac{-1}{J_k}) = -1$.
As in the proof of Lemma~\ref{divislemma},
applying Lemma~\ref{lemma:recurrence} to the odd
primes $p=3,5,17,37$ that can divide $a$, we found that the periods $m_p$ of the sequences $J_k\bmod p$ are $m_3=8$, $m_5=24$, $m_{17}=144$, and $m_{37}=36$.
Since $(\frac{-1}{J_k}) = -1$, it follows from quadratic reciprocity that for $a=-5,-17$, and $-111$, the period of the sequence $(\frac{a}{J_k})$ divides the least common multiple of the periods $m_p$ for $p|a$. For $a=-6$, by Lemma~\ref{Jacobicomputations}(ii) the period of $(\frac{2}{J_k})$ is 2, which already divides $m_3=8$.
Since $3$ is the period of the sequence $(\frac{j_k}{\sqrt{-7}})$, we find
the period $m$ of $(\frac{a}{J_k})(\frac{j_k}{\sqrt{-7}})$ listed in Table~\ref{table:Sa}
by taking the least common multiple of 3 and the $m_p$ for $p|a$.
To compute $S_a$, it then suffices to compute $(\frac{a}{J_k})$ and check when 
$(\frac{a}{J_k})=(\frac{j_k}{\sqrt{-7}})$, for $1 < k\le m+1$.
\end{proof}

\subsection{The set $\boldsymbol{T_a}$}

We now define the sets $T_a$.

\begin{definition}
\label{TaDef}
Let $a$ be a squarefree integer, and suppose that $P\in E_a(K)$.
Then the field $K(\alpha^{-1}(P))$ has degree $1$ or $2$ over $K$, so it can be written in the form
$K(\sqrt{\delta_{P}})$ with $\delta_{P}\in K$.
Let 
$$
\textstyle{T_{P}  := \bigl\{ k > 1 : \left(\frac{\delta_{P}}{j_k}\right) = -1 \bigr\}.}
$$
For the values of $a$ listed in Table~\ref{table:twistspts}, let
$T_a=T_{P_a}$ and let $\delta_a=\delta_{P_a}$.
\end{definition}

\begin{lemma}
\label{generatorprop}
Suppose that $k > 1$, $j_k$ is prime in $\OK$, and $a$ is a squarefree integer.
Suppose that $P\in E_a(K)$, and let $\tilde{P}$ denote the reduction of $P$ mod $j_k$.
Then 
$\tilde{P} \not\in \alpha E_{a}(\OK/(j_k))$ if and only if 
$k \in T_{P}$.
\end{lemma}

\begin{proof}
Let $L=K(\alpha^{-1}(P)) = K(\gamma)$ for some $\gamma \in L$ such that 
$\gamma^2 = \delta_{P}$.
Fix a $Q \in E_{a}(\bar{\Q})$ such that $\alpha Q = P$.
Since $\ker(\alpha) \subset E_{a}[2] \subset E_{a}(K)$, we have
$K(Q)=L= K(\gamma)$. 
Fix a prime ideal $\frakp$ of $L$ above $(j_k)$,
let $\F = \OK/(j_k)$,
let $\tilde{Q} \in E_{a}(\bar{\F})$ be the reduction of $Q$ mod $\frakp$,
and let $\tilde{\gamma}$ be the reduction of $\gamma$ mod $\frakp$.
Then $\F(\tilde{Q}) = \F(\tilde{\gamma})$.

Now $\tilde{P} \in \alpha E_{a}(\F)$ if and only if
$\tilde{Q} \in E_a(\F)$. By the above, this happens if and only if $\tilde{\gamma} \in \F$, that is,
if and only if $\delta_{P}$ is a square modulo $j_k$.
\end{proof}

\begin{lemma}
\label{deltacomputation}
We can take
$$
\delta_{-1} = \alpha, \quad \delta_{-5} = -5\alpha, \quad \delta_{-6} = -3\sqrt{-7},\quad  \delta_{-17} =  \alpha,
\quad  \delta_{-111} = -3.
$$
\end{lemma}

\begin{proof}
The action of the endomorphism $\alpha$ on the elliptic curve $E_a$
and its reductions is as follows
(see Proposition II.2.3.1 of \cite[p.\ 111]{silverman94}).  For $(x,y)\in E_a$, we have
$$
\alpha(x,y) =
\textstyle{\left(\frac{2x^2 + a(7 - \sqrt{-7})x + a^2(-7 - 21\sqrt{-7})}{(-3+\sqrt{-7})x + a(-7+5\sqrt{-7})}, \frac{y\left(2x^2 + a(14-2\sqrt{-7})x + a^2(28+14\sqrt{-7})\right)}{-(5+\sqrt{-7})x^2 - a(42+2\sqrt{-7})x - a^2(77 - 7\sqrt{-7})}\right).}
$$
Solving for $R$ in $\alpha R=P_a$ yields $\delta_a$ in each case.
\end{proof}

\begin{lemma}
\label{alpha_legendre}
If $k > 1$ then $\left(\frac{\alpha}{j_k}\right) = -1$.
\end{lemma}

\begin{proof}
Let $M = K\left(\sqrt{\alpha}\right)$. By 
the reciprocity law of global class field theory we have
$$
\prod_{\frakp} \left(j_k, M_{\frakp}/K_{\frakp}\right) = 1,
$$
where 
$\left(j_k, M_{\frakp}/K_{\frakp}\right)$ is the norm residue symbol. 

Let $f(x) = x^2 - j_k\in\O_{K_{\alpha}}[x]$. For $k>1$ we have
$$
|f(1)|_{\alpha} = \left|2\alpha^k\right|_{\alpha} = 2^{-(k+1)} < 2^{-2} = |4|_{\alpha} = \left|f'(1)^2\right|_{\alpha},
$$
and Hensel's lemma implies that $f(x)$ has a root in $\O_{K_{\alpha}}$.
Thus $j_k$ is a square in~$K_{\alpha}$ and $\left(j_k, M_{\alpha}/K_{\alpha}\right) = 1$.

Identify $K_{\bar{\alpha}}$ with $\QQ_2$. 
Applying Theorem 1 of \cite[p.\ 20]{serre73}
with $a = j_k$ and $b = \alpha$, 
and using $\bar{\alpha}^5 = 5 + \alpha$,
gives
$(j_k,\alpha) = -1$,
where $(j_k,\alpha)$ is the Hilbert symbol. 
Thus $j_k \not\in \norm_{M_{\bar{\alpha}}/K_{\bar{\alpha}}}\left(M_{\bar{\alpha}}^*\right)$, and therefore
$\left(j_k, M_{\bar{\alpha}}/K_{\bar{\alpha}}\right) = -1$.

If $\frakp$ is a prime ideal of $\OK$ that does not divide $2$, then
$M_{\frakp}/K_{\frakp}$ is unramified.  By local class field theory we then have
$$
(j_k,M_{\frakp}/K_{\frakp}) = \left(\frac{\alpha}{\frakp}\right)^{\ord_\frakp(j_k)}.
$$
Since  $j_k$ is prime to $2$, we have $\ord_\alpha(j_k) = \ord_{\bar{\alpha}}(j_k) = 0$, hence
$$
\prod_{\frakp\nmid 2} \left(j_k, M_{\frakp}/K_{\frakp}\right) = 
\prod_{\frakp\nmid 2} \left(\frac{\alpha}{\frakp}\right)^{\ord_\frakp(j_k)} = 
\prod_{\text{all }\frakp} \left(\frac{\alpha}{\frakp}\right)^{\ord_\frakp(j_k)} = 
\left(\frac{\alpha}{j_k}\right).
$$
Therefore,
$$
1 = \prod_{\frakp} \left(j_k, M_{\frakp}/K_{\frakp}\right)
	 = \left(\frac{\alpha}{j_k}\right)(j_k, M_{\alpha}/K_{\alpha})(j_k, M_{\bar{\alpha}}/K_{\bar{\alpha}})
	= -\left(\frac{\alpha}{j_k}\right),
$$
as desired.
\end{proof}

\begin{lemma}
\label{TaLem}
For $a\in\{-1,-5,-6,-17,-111\}$ the sets $T_a$ are as follows:
\smallskip

\begin{center}
\begin{tabular}{lll}
$T_{-1}$&$=$&$\{ k > 1 \},$\\
$T_{-5}$&$=$&$\{ k > 1 : k \equiv 3,4,7,8,11,13,14,15,16,17,20,22 \pmod{24} \},$\\
$T_{-6}$&$=$&$\{ k > 1 : k \equiv 1,5,10,12,15,19,20,21,22,23 \pmod{24} \}$,\\
$T_{-17}$&$=$&$\{ k > 1 \}$,\\
$T_{-111}$&$=$&$\{ k > 1 : k \equiv 1, 2, 3, 6 \pmod{8} \}$.
\end{tabular}
\end{center}
\end{lemma}
\begin{proof}
We  apply Lemma \ref{deltacomputation} and the definition of $T_a$.
Lemma \ref{alpha_legendre} implies that $T_{-1} = T_{-17} =\{ k > 1 \}$.
For $a=-6$ we use quadratic reciprocity in quadratic fields
(see Theorem 8.15 of \cite[p.\ 257]{lemmermeyer}) to compute 
$\left(\frac{\sqrt{-7}}{j_k}\right)$.
For the remaining cases we compute
$\left(\frac{-3}{j_k}\right)=\left(\frac{-3}{J_k}\right)$ and
$\left(\frac{-5}{j_k}\right)=\left(\frac{-5}{J_k}\right)$ 
as in the proof of
Lemma~\ref{structure2}, and apply $\left(\frac{\alpha}{j_k}\right)=-1$ from Lemma~\ref{alpha_legendre}.
\end{proof}

\subsection{Proof of Theorem \ref{mainthm}}

\begin{lemma}
\label{generator}
Let $a$ be a squarefree integer. 
Suppose that $P\in E_a(K)$, $k \in S_a \cap T_P$, and
$j_k$ is prime.
Let $\tilde{P}$ denote the reduction of $P$ mod $j_k$.
Then the annihilator of $\tilde{P}$ in $\OK$ is divisible by $\alpha^{k+1}$.
\end{lemma}

\begin{proof}
We have $E_a(\OK/(j_k))\simeq\OK/(2\alpha^k)=\OK/(\overline{\alpha}\alpha^{k+1})$, by Lemma \ref{structure1}(ii).
It then suffices to show $\tilde{P} \not\in \alpha E_{a}(\OK/(j_k))$, which follows from Lemma~\ref{generatorprop}.
\end{proof}

The congruence conditions for $k$ in Table \ref{table:twistspts} come from taking $S_a \cap T_a$,
excluding the cases handled by Lemma \ref{divislemma}, and adjusting to give disjoint sets.

We now prove Theorem \ref{mainthm}.
Suppose that $k>1$, $k \not\equiv 0 \pmod{8}$,  $k \not\equiv 6 \pmod{24}$, 
and $J_k$ is prime.  Let $a$ and $P_a$ be as listed in Table~\ref{table:twistspts}.
Then $k\in S_a\cap T_a$.  Let $\tilde{P}$ denote the reduction of $P_a$ mod $j_k$.
We have $E_{a}(\OK/(j_k)) \simeq \OK/(2\alpha^{k})$ by Lemma \ref{structure1}(ii),
and therefore the annihilator of $\tilde{P}$ in $\OK$ divides~$2\alpha^{k}$.
By Lemma \ref{generator}, 
the annihilator of $\tilde{P}$ in $\OK$ is divisible by $\alpha^{k+1}$.
Since $2\alpha^k$ divides~$2^{k+1}$ but $\alpha^{k+1}$ does not divide $2^{k}$, 
we must have $2^{k+1}\tilde{P} = 0$ and $2^{k}\tilde{P} \neq 0$.
Therefore $2^{k+1}P_a$ is zero mod ${J_k}$ and $2^kP_a$ is strongly nonzero mod $J_k$. 

For the converse, note that $\disc(E_a) = -2^{12}\cdot 7^3\cdot a^6$,
so Lemma \ref{divislemma} shows that
$\gcd(J_k,\disc(E_a))=1$ if $k \not\equiv 0 \pmod{8}$ and $k \not\equiv 6 \pmod{24}$.
We can therefore apply Proposition~\ref{prop:ECPP2} with 
$m=2^{k+1}$, noting that
$$2^{k+1} > ((3\cdot 2^{k+1})^{\frac{1}{4}}+1)^2 > (J_k^{1/4} + 1)^2$$
for all $k >2$, and for $k=2$ we have $2^{k+1}=8 > (11^{1/4}+1)^2=(J_k^{1/4}+1)^2$.
This proves Theorem \ref{mainthm}.

\begin{rem}
As pointed out by Richard Pinch, $P_a\in 2E_{a}(\OK/(j_k))$ if and only if
all $x(P_a)-e_i$ are squares mod $j_k$, where $E_a$ is $y^2=\prod_{i=1}^3(x-e_i)$
and $x(P_a)$ is the $x$-coordinate. 
We  tested for divisibility by $\alpha$ instead of by $2$,
to make it clearer how this approach (as initiated by Gross in \cite{gross04}) 
makes use of the $\OK$-module structure of $E_{a}(\OK/(j_k))$.
Such an approach is useful for further generalizations.
\end{rem}

\section{Algorithm}\label{section:computations}

A na\"ive implementation of Theorem~\ref{mainthm} is entirely straightforward, but here we describe a particularly efficient implementation and analyze its complexity.
We then discuss how the algorithm may be used in combination with sieving to search for prime values of $J_k$, and give some computational results.

\subsection{Implementation}

There are two features of the primality criterion given by Theorem~\ref{mainthm} worth noting.
First, it is only necessary to perform the operation of adding a point on the elliptic curve to itself (doubling), no general additions are required.
Second, testing whether a projective point $P=[x,y,z]$ is zero or strongly nonzero modulo an integer $J_k$ only involves the $z$-coordinate:  $P$ is zero mod $J_k$ if and only if $J_k|z$, and $P$ is strongly nonzero mod $J_k$ if and only if $\gcd(z,J_k)=1$.

To reduce the cost of doubling, we transform the curve
\[
E_a\colon\qquad y^2=x^3-35a^2x-98a^3
\]
to the Montgomery form \cite{montgomery87}
\[
E_{A,B}\colon\qquad By^2 = x^3 + Ax^2+x.
\]
Such a transformation is not possible over $\Q$, but it can be done over $\Q(\sqrt{-7})$.
In general, one transforms a short Weierstrass equation $y^2=f(x)=x^3+a_4x+a_6$ into Montgomery form by choosing a root $\gamma$ of $f(x)$ and setting $B=(3\gamma^2+a_4)^{-1/2}$ and $A=3\gamma B$; see, e.g., \cite{khk00}.
For the curve $E_a$, we choose $\gamma=\frac{1}{2}(-7+\sqrt{-7})a$, yielding
$$
A=\frac{-15-3\sqrt{-7}}{8}\qquad\text{and}\qquad B=\frac{7+3\sqrt{-7}}{56a}.
$$
With this transformation, the point $P_a=(x_0,y_0)$ on $E_a$ corresponds to the point $(B(x_0-\gamma), By_0)$ on the Montgomery curve $E_{A,B}$, and is defined over $\Q(\sqrt{-7})$.

In order to apply this transformation modulo $J_k$, we need a square root of $-7$ in $\Z/J_k\Z$.
If $J_k$ is prime and $d = 7^{(J_k+1)/4}$, then 
$$d^2\equiv 7^{(J_k-1)/2}\cdot 7 \equiv \left(\frac{7}{J_k}\right)7 \equiv -7\pmod{J_k},$$
since $J_k\equiv 3 \pmod 4$ and $J_k\equiv 2,4\pmod 7$ is a quadratic residue modulo 7.
If we find that $d^2\not\equiv -7\pmod{J_k}$, then we immediately know that $J_k$ must be composite and no further computation is required.

With the transformation to Montgomery form, the formulas for doubling a point on $E_a$ become particularly simple.
If $P=[x_1,y_1,z_1]$ is a projective point on $E_{A,B}$ and $2P=[x_2,y_2,z_2]$, we may determine $[x_2,z_2]$ from $[x_1,z_1]$ via
\begin{align}\label{doubling formulas}
4x_1z_1 &= (x_1+z_1)^2-(x_1-z_1)^2,\\\notag
x_2 &= (x_1+z_1)^2(x_1-z_1)^2,\\\notag
z_2 &= 4x_1z_1\bigl((x_1-z_1)^2 + C(4x_1z_1)\bigr),\notag
\end{align}
where
$$
C = (A+2)/4 = \frac{1-3\sqrt{-7}}{32}.
$$
Note that $C$ does not depend on $P$ (or even $a$), and may be precomputed.
Thus doubling requires just 2 squarings, 3 multiplications, and 4 additions in $\Z/J_k\Z$.

We now present the algorithm, which exploits the transformation of $E_a$ into Montgomery form.
We assume that elements of $\Z/J_k\Z$ are uniquely represented as integers in $[0,J_k-1]$.
\bigskip

\label{algorithm}
\addtocounter{theorem}{1}
\noindent
\textbf{Algorithm~\ref{algorithm}}\\
\textbf{Input}: positive integers $k$ and $J_k$.\\
\textbf{Output}: \texttt{true} if $J_k$ is prime and \texttt{false} if $J_k$ is composite.\\
\vspace{-12pt}
\begin{enumerate}[\bf 1.]
\setlength\itemindent{20pt}
\setlength\itemsep{4pt}
\item If $k\equiv 0\pmod{8}$ or $k\equiv 6\pmod{24}$ then return \texttt{false}.
\item Compute $d=7^{(J_k+1)/4}\bmod J_k$.
\item If $d^2\not\equiv -7 \pmod{J_k}$ then return \texttt{false}.
\item Determine $a$ via Table~\ref{table:twistspts}, depending on $k\pmod{72}$.
\item Compute $r=(-7+d)a/2\bmod J_k$, $B=(7+3d)/(56a)\bmod J_k$, and\\\phantom{\hspace{20pt}}$C=(1-3d)/32\bmod J_k$.
\item Let $x_1=B(x_0-r)\bmod J_k$ and $z_1=1$, where $P_a=(x_0,y_0)$ is as in Table~\ref{table:twistspts}.
\item For $i$ from 1 to $k+1$, compute $[x_{i},z_{i}]$ from $[x_{i-1},z_{i-1}]$ via (\ref{doubling formulas}).
\item If $\gcd(z_k,J_k) = 1$ and $J_k|z_{{k+1}}$ then return \texttt{true}, otherwise return \texttt{false}.
\end{enumerate}
\bigskip

The tests in step 1 rule out cases where $J_k$ is divisible by 3 or 5, by Lemma~\ref{divislemma}; $J_k$ is then composite, since $J_k>5$ for all $k$.
This also ensures $\gcd(a,J_k)=1$
(see Lemma~\ref{divislemma}), so the divisions in step 5 are all valid ($J_k$ is never divisible by 2 or 7).
By Remark~\ref{rem:nonzero}, for $k\ge 6$ the condition $\gcd(z_k,J_k) = 1$ in step 8 can be replaced with $z_k\not\equiv 0\bmod J_k$.

\begin{proposition}\label{prop:complexity}
Algorithm~\ref{algorithm} performs $6k+o(k)$ multiplications and $4k$ additions in $\Z/J_k\Z$.
Its time complexity is $O(k^2\log k\log\log k)$ and it uses $O(k)$ space.
\end{proposition}
\begin{proof}
Using standard techniques for fast exponentiation \cite{yao76}, step 2 uses $k+o(k)$ multiplications in $\Z/J_k\Z$.
Steps 5-6 perform $O(1)$ operations in $\Z/J_k\Z$ and step 7 uses $5k$ multiplications and $4k$ additions.
The cost of the divisions in step 5 are comparatively negligible, as is the cost of step 8.
Multiplications (and additions) in $\Z/J_k\Z$ have a bit complexity of $O(\textsf{M}(k))$, where $\textsf{M}(k)$ counts the bit operations needed to multiply two $k$-bit integers \cite[Thm.~9.8]{gg}.
The bound on the time complexity of Algorithm~\ref{algorithm} then follows from the Sch\"onhage-Strassen \cite{ssmult} bound: $\textsf{M}(k)=O(k\log k\log\log k)$.
The space complexity bound is immediate: the algorithm only needs to keep track of two pairs $[x_i,z_i]$ and $[x_{i-1},z_{i-1}]$ at any one time, and elements of $\Z/J_k\Z$ can be represented using $O(k)$ bits.
\end{proof}

\begin{table}[ht]
\caption{Timings for Algorithm~\ref{algorithm}}\label{table:timings}
\vspace{-12pt}
{\footnotesize (CPU seconds on a 3.0 GHz AMD Phenom II 945)}
\bigskip

\begin{tabular}{lrr}
$k$ & \hspace{24pt}step 2 & \hspace{30pt}step 7\\\hline
&&\vspace{-8pt}\\
$2^{10}+1$ & 0.00 & 0.01\\
$2^{11}+1$ & 0.00 & 0.02\\
$2^{12}+1$ & 0.02 & 0.15\\
$2^{13}+1$ & 0.15 & 0.91\\
$2^{14}+1$ & 0.88 & 5.50\\
$2^{15}+1$ & 5.26 & 32.2\\
$2^{16}+1$ & 27.5 & 183\\
$2^{17}+1$ & 133 & 983\\
$2^{18}+1$ & 723 & 5010\\
$2^{19}+1$ & 3310 & 23600 \\
$2^{20}+1$ & 13700 & 107000\\\hline
\end{tabular}
\end{table}

Table \ref{table:timings} gives timings for Algorithm~\ref{algorithm} when implemented using the \texttt{gmp} library for all integer arithmetic, including the gcd computations.
We list the times for step~2 and step~7 separately (the time spent on the other steps is negligible).
In the typical case, where $J_k$ is composite, the algorithm is very likely\footnote{Indeed, we have yet to encounter even a single $J_k$ that is a strong pseudoprime base~$-7$.} to terminate in step 2, which effectively determines whether $J_k$ is a strong probable prime base~$-7$, as in  \cite[Alg.~3.5.3]{crandal05}.
To obtain representative timings at the values of $k$ listed, we temporarily modified the algorithm to skip step 2.

We note that the timings for step 7 are suboptimal due to the fact that we used the \texttt{gmp} function \texttt{mpz\underline\  mod} to perform modular reductions.  A lower level implementation (using Montgomery reduction \cite{montgomery85}, for example) might improve these timings by perhaps 20 or 30 percent.  

We remark that Algorithm~\ref{algorithm} can easily be augmented, at essentially no additional cost, to retain an intermediate point $Q=[x_s,y_s,z_s]$,
where $s=k+1-r$ is chosen so that the order $2^r$ of $Q$ is the least power of $2$ greater than $(J_k^{1/4}+1)^2$.
The value of $y_s$ may be obtained as a square root of $y_s^2=(x_s^3+Ax_s^2z_s+x_sz_s^2)/(Bz_s)$ by computing $(y_s^2)^{(J_k+1)/4}$.
When $J_k$ is prime, the algorithm can then output a Pomerance-style certificate $(E_{A,B},Q,r,J_k)$ for the primality of $J_k$.
This certificate has the virtue that it can be verified using just $2.5k+O(1)$ multiplications in $\Z/J_k\Z$, versus the $6k+o(k)$ multiplications used by Algorithm~\ref{algorithm}, by checking that the point $Q$ has order $2^r$ on the elliptic curve $E_{A,B}\bmod J_k$.

\begin{table}[ht]
\caption{Prime values of $J_k\approx 2^{k+2}$ for $k\le 1.2\times 10^6$.}\label{table:primes}
\begin{tabular}{rrrrrrrrrrrrr}
$k$&$J_k$&$a$&&&$k$&$J_k$&$a$&&&$k$&$J_k$&$a$\\\hline
    2 &        11 & -1 &&&   319 & 427...247 & -5 &&& 17807 & 110...799 & -1\\
    3 &        23 & -1 &&&   375 & 307...023 & -1 &&& 18445 & 125...407 & -5\\
    4 &        67 & -5 &&&   467 & 152...727 & -1 &&& 19318 & 793...763 & -5\\
    5 &       151 & -1 &&&   489 & 639...239 & -1 &&& 26207 & 495...799 & -1\\
    7 &       487 & -5 &&&   494 & 204...963 & -1 &&& 27140 & 359...907 & -1\\
    9 &      2039 & -1 &&&   543 & 115...143 & -1 &&& 31324 & 116...867 & -5\\
   10 &      4211 & -6 &&&   643 & 145...399 & -17 &&& 36397 & 155...007 & -5\\
   17 &    524087 & -1 &&&   684 & 321...531 & -1 &&& 47294 & 327...963 & -1\\
   18 &   1046579 & -1 &&&   725 & 706...551 & -1 &&& 53849 & 583...567 & -1\\
   28 & 107...427 & -5 &&&  1129 & 291...591 & -17 &&& 83578 & 122...491 & -6\\
   38 & 109...043 & -1 &&&  1428 & 297...011 & -1 &&& 114730 & 593...411 & -6\\
   49 & 225...791 & -17 &&&  2259 & 425...023 & -1 &&& 132269 & 345...831 & -1\\
   53 & 360...711 & -1 &&&  2734 & 415...123 & -5 &&& 136539 & 864...023 & -1\\
   60 & 461...451 & -1 &&&  2828 & 822...787 & -1 &&& 147647 & 599...399 & -1\\
   63 & 368...943 & -1 &&&  3148 & 175...227 & -5 &&& 167068 & 120...027 & -5\\
   65 & 147...007 & -1 &&&  3230 & 849...483 & -1 &&& 167950 & 388...883 & -5\\
   77 & 604...191 & -1 &&&  3779 & 156...127 & -1 &&& 257298 & 104...179 & -1\\
   84 & 773...531 & -1 &&&  5537 & 254...887 & -1 &&& 342647 & 423...399 & -1\\
   87 & 618...703 & -1 &&&  5759 & 171...279 & -1 &&& 414349 & 120...207 & -5\\
  100 & 507...507 & -5 &&&  7069 & 382...207 & -5 &&& 418033 & 118...831 & -17\\
  109 & 259...207 & -5 &&&  7189 & 508...207 & -5 &&& 470053 & 451...407 & -5\\
  147 & 713...023 & -1 &&&  7540 & 233...107 & -5 &&& 475757 & 536...791 & -1\\
  170 & 598...611 & -1 &&&  7729 & 183...591 & -111 &&& 483244 & 347...667 & -5\\
  213 & 526...239 & -1 &&&  9247 & 168...687 & -5 &&& 680337 & 279...759 & -1\\
  235 & 220...519 & -17 &&& 10484 & 398...747 & -1 &&& 810653 & 295...711 & -1\\
  287 & 994...999 & -1 &&& 15795 & 234...023 & -1 &&& 857637 & 115...519 & -1\\
&&&&&&&&&& 1111930 & 767...411 & -6\\\hline
\end{tabular}
\end{table}

\subsection{Searching for prime values of $J_k$}
While one can directly apply Algorithm~\ref{algorithm} to any particular $J_k$,
when searching a large range $1\le k\le n$ for prime values of $J_k$ it is more efficient to first \emph{sieve} the interval $[1,n]$ to eliminate values of $k$ for which $J_k$ cannot be prime.

For example, as noted in Lemma \ref{divislemma}, if $k\equiv 0\pmod 8$ then $J_k$ is divisible by~$3$.
More generally, for any small prime $\ell$, one can very quickly compute $J_k\bmod \ell$ for all $k\le n$ by applying the linear recurrence \eqref{recurrence} for $J_k$, working modulo~$\ell$.
If $\ell < \sqrt{n}$, then the sequence $J_k\bmod \ell$ will necessarily cycle, but in any case it takes very little time to identify all the values of $k\le n$ for which $J_k$ is divisible by~$\ell$; the total time required is just $\tilde{O}(n\log\ell)$, versus $\tilde{O}(n^2)$ if one were to instead apply a trial division by~$\ell$ to each $J_k$.

We used this approach to sieve the interval $[1,n]$ for those $k$ for which $J_k$ is not divisible by any prime $\ell\le L$.
Of course one still needs to consider $J_k\le L$, but this is a small set consisting of roughly $\log_2 L$ values, each of which can be tested very quickly.
With $n=10^6$ and $L=2^{35}$, sieving reduces the number of potentially prime $J_k$ by a factor of more than 10, leaving 93,707 integers $J_k$
as candidate primes to be tested with Algorithm~\ref{algorithm}.
The prime values of $J_k$ found by the algorithm are listed in Table~\ref{table:primes}, along with the corresponding value of $a$.
As noted in the introduction, we have extended these results to $n=1.2\times 10^6$, finding one additional prime with $k=$ 1,111,930, which is also listed in Table 4.

The data in Table~\ref{table:primes} suggests that prime values of $J_k$ may be more common than prime values of Mersenne numbers $M_n$; there are 78 primes $J_k$ with fewer than one million bits, but only 33 Mersenne primes in this range.
This can be at least partly explained by the fact that $M_n$ can be prime only when $n$ is prime, whereas the values of $k$ for which $J_k$ can be prime are not so severely constrained.  By analyzing these constraints in detail, it may be possible to give a heuristic estimate for the density of primes in the sequence $J_k$, but we leave this to a future article.

\end{document}